\documentclass[12pt, a4paper]{article}
\usepackage{custom}
\usepackage{float}

\title{Local models of isolated singularities for affine special K\"ahler structures in dimension two}
\author{Martin Callies{\  } and Andriy Haydys}
\date{19th June 2018}

\begin{document}
\maketitle

 \begin{abstract}
  We construct local models of isolated  singularities for special K\"ahler structures in real dimension two assuming that the associated holomorphic cubic form does not have essential singularities. 
  As an application we compute the holonomy of the flat symplectic connection, which is a part of the special K\"ahler structure. 
 \end{abstract}


\section{Introduction}

The notion of a special K\"ahler structure appeared for the first time in physics~\cite{Gates84_SuperspaceFormulation_SpKaehler, deWitVanProyen84PotentialsSymmetries} and was formalized by Freed~\cite{Freed99_SpecialKaehler}.
For reader's convenience, let us recall the definition of the affine special K\"ahler structure. 

\begin{definition}
	An \emph{(affine) special K\"ahler structure}\index{metric!special K\"ahler} on a manifold $M$ is a quadruple $(g,I,\omega, \nabla)$, where $(M,g,I,\omega)$ is a K\"ahler manifold with Riemannian metric $g$, complex structure $I$, and symplectic form $\omega(\cdot,\cdot)=g(I\cdot,\cdot)$, and $\nabla$ is a flat symplectic torsion-free connection on the tangent bundle $TM$ such that 
	\begin{equation}
	\label{Eq_SpKaehlerCondition}
	(\nabla^{}_XI)Y=\, (\nabla^{}_Y I)X
	\end{equation}
	holds for all vector fields $X$ and $Y$.
\end{definition}

If $I$ is fixed, which is always assumed to be the case below, we say for simplicity that $(g,\nabla)$ is a special K\"ahler structure.

The importance of special K\"ahler structures stems from the so called c-map construction, which associates to each special K\"ahler manifold a \hK one, which is equipped with the structure of a holomorphic Lagrangian fibration~\cites{CecottiEtAl89_GeomOfTypeIISuperstrings, Freed99_SpecialKaehler, MaciaSwann15_TwistGeomCmap}.
Moreover, under suitable hypotheses the converse construction also exists~\cite{Freed99_SpecialKaehler}*{Sect.\,3}.
Typically, a holomorphic Lagrangian fibration contains singular fibers and in this case the corresponding special K\"ahler structure has  singularities.   
The importance of \emph{singular} special K\"ahler structures can be also seen from the following fact~\cite{Lu99_NoteOnSpKaehler, BauesCortes01_SpKParabolicSpheres}: A complete special K\"ahler metric is necessarily flat.  

Isolated singularities of affine special K\"ahler structures in the simplest case of two real dimensions are in the focus of this article, which is closely related to~\cite{Haydys15_IsolSing_CMP}.
However, here we focus on the properties of  $\nabla$ rather than $g$ near an isolated singularity. 
Properties of $\nabla$, in particular its holonomy, contain important information about behavior of the corresponding holomorphic Lagrangian fibration near a singular fiber.

A number of special K\"ahler structures with isolated singularities can be found~\cite{Haydys15_IsolSing_CMP, CalliesHaydys17_AffSpK2D_Arx}.  
We now describe a family of examples, whose significance will be clear below. 
Thus, let $B_1^*:=\{   0< |z|<1 \}$ be the punctured disc; 
Denote by $(r, \theta)$ the polar coordinates on $\C^*$, where $\theta\in (0,2\pi)$, and  put $\rho:=\log r$,
\[
\mathbbm 1_2:= 
\begin{pmatrix}
1 & 0\\
0 & 1
\end{pmatrix},\qquad\text{and}\qquad
\rI_2:=
\begin{pmatrix}
0 & -1\\
1 & \phantom{-}0
\end{pmatrix}.
\]
We show that for any $k\in\Z$, $C\in \R_{>0}$,  and $b\in\C, \ |b|=1,$ the following
\begin{equation}
 \label{Eq_ModelLogSing}
	 \begin{aligned}
	 g_{k } &= -C\, r^{k}\log r\, |dz|^2,\\
	 \omega_{k,\nabla} &= 
	 \frac 12
	 \left(k\rI_2 + 
	 \begin{pmatrix} 
	 \Im\!\left(b e^{ik\theta}\right) & -1+\Re\!\left(b e^{ik\theta}\right)\\
	 1 + \Re\!\left(b e^{ik\theta}\right) &  -\Im\!\left(b e^{ik\theta}\right)
	 \end{pmatrix} \rho^{-1} \right) d\theta  \\
	 &+  
	 \frac 12\left(k\mathbbm 1_2	+
	 \begin{pmatrix}
	 1 - \Re\!\left(b e^{ik\theta}\right) & \Im\!\left(b e^{ik\theta}\right)\\
	 \Im\!\left(b e^{ik\theta}\right) & 1+\Re\!\left(b e^{ik\theta}\right)
	 \end{pmatrix}\rho^{-1}  
	 \right) d\rho
	 \end{aligned}
\end{equation}
is a special K\"ahler structure on $B_1^*$, where the dependence on $C$ and $b$ is suppressed in the notations.  
Here $\om_{k,\nabla}$ is the connection one-form of $\nabla$ \wrt the trivialization $(\partial_x, \partial_y)$, where $z=x+ yi$; 
See Section~\ref{Sect_Models} for further details.
Even though~\eqref{Eq_ModelLogSing} may look incomprehensible at first glance, we show that many objects of interest, such as special holomorphic coordinates, the associated holomorphic cubic form, and the holonomy around the origin for these special K\"ahler structures can be computed explicitly. 

While for some values of $k$ the metric $g_k$ was previously known to be special K\"ahler, see for instance~\cite{CalliesHaydys17_AffSpK2D_Arx}*{Ex.\,23,\,24}, we believe for most values of $k$ these examples are new. 

Furthermore, we show that~\eqref{Eq_ModelLogSing} together with flat cones
\begin{equation}
  \label{Eq_FlatConicalStr}
g^c_\b = r^{\b} |dz|^2,\qquad \omega^c_{\b,\nabla}=\omega_{LC} = \frac \b 2
\bigl ( \mathbbm 1_2\, d\rho + \rI_2\, d\theta \bigr),
\end{equation}
 where $\b\in\R$, are local models of isolated singularities of affine special K\"ahler structures in real dimension two. 
 To explain, let $(M, I)$ be a smooth Riemann surface and let $(g,\nabla)$ be  a special K\"ahler structure on $M$ possibly singular at some $m_0\in M$.
 Assume  that the  associated holomorphic cubic form $\Xi$ has a finite order $n\in\Z$ at $m_0$. 
Choosing a local holomorphic coordinate $z$ near $m_0$, we can write  $g = w\,|dz|^2$ and $\nabla = d +\om_\nabla$. 
By the main theorem  of~\cite{Haydys15_IsolSing_CMP} there are two possibilities:  
$ w = -|z|^{n+1}\log |z| (C+ o(1))$ or there is $\b<n+1$ such that $w=|z|^\b (C + o(1))$.
In this article we prove the following result.
\begin{thm}
	\label{Thm_spKLocalModels}
	Let $(g,\nabla)$ be a special K\"ahler structure on $M$ possibly singular at $m_0$, where $(M,I)$ is a smooth Riemann surface. Assume that $\Xi$ has a finite order  $n$ at $m_0$.
	Let $z$ be a local holomorphic coordinate near $m_0$.
	\begin{enumerate} 
		\item If $w=r^\b\bigl (C+o(1)\bigr )$, where $\b<n+1$, 
		then there exists $\e>0$ such that  
		\[ 
		\omega_\nabla = \om^c_{\b, \nabla} + o(1)\,d\theta + o(e^{\e\rho})\, d\rho\qquad\mathrm{ for }~\rho\to -\infty.
		\]
		
		\item If $w=-r^{n+1}\log r\bigl (C+o(1)\bigr )$, then there exists $b\in\C$, $|b|=1$,  such that 
		\[ 
		\begin{aligned}
		\omega_\nabla &= \om_{n+1,\nabla} + o(\rho^{-1})\, d\theta + o(\rho^{-1})\, d\rho\qquad \text{ for }~\rho\to -\infty.
		\end{aligned}
		\]
	\end{enumerate}
\end{thm}

Motivated by~\autoref{Thm_spKLocalModels} we adopt the following terminology.

\begin{defn}
	We say that a special K\"ahler structure $(g,\nabla)$ on $B_1^*$ has \emph{a conical singularity} of order $\tfrac 12\b$ at the origin, if $(g,\om_\nabla)$ is asymptotic to $(g^c_\b, \om^c_{\b,\nabla})$.  
	We say that $(g,\nabla)$ has \emph{a logarithmic singularity} of order $\tfrac 12 k,\ k\in \Z$, at the origin, if  $(g,\om_\nabla)$ is asymptotic to $(g_{k}, \om_{k,\nabla})$.  
\end{defn}

Theorem~\ref{Thm_spKLocalModels} yields in particular the following result, which has been announced in~\cite{CalliesHaydys17_AffSpK2D_Arx}.

\begin{corollary}
	\label{Cor_Holonomy}
	Let $(g,\nabla)$ be as in Theorem~\ref{Thm_spKLocalModels}.
	If $(g,\nabla)$ has a logarithmic singularity of order $\tfrac 12(n+1)$,  we put by definition $\b=n+1$.
	Denote by $\Hol(\rS^1,\nabla)$ the holonomy of $\nabla$ along a circle centered at  the origin and oriented in the counterclockwise direction. 
	Then the following holds:
	\begin{itemize}[itemsep=2pt]
		\item If $\b\notin\Z$, then $\Hol(\rS^1, \nabla)$ is conjugate to $\begin{pmatrix}
		\cos\pi\b & -\sin\pi\b\\
		\sin\pi\b & \cos\pi\b
		\end{pmatrix}$;
		\item If $\b\in 2\Z$, then $\Hol(\rS^1, \nabla)$ is trivial or conjugate to $\begin{pmatrix}
		1 & 1\\
		0 & 1
		\end{pmatrix}$;
		\item If $\b\in 2\Z+1$, then $\Hol(\rS^1,\nabla)$ is $-\mathbbm 1$ or conjugate to $\begin{pmatrix}
		-1 & \phantom{-}1\\
		0 & -1
		\end{pmatrix}$.
	\end{itemize} 
	Moreover, if $(g,\nabla)$ has a conical singularity of  order $\tfrac 12\b$, where $\b\in\Z$,  then $\Hol (\rS^1, \nabla)=\mathbbm 1$ if $\b$ is even and $\Hol (\rS^1, \nabla)=-\mathbbm 1$ if $\b$ is odd.
\end{corollary}

Proofs of these statements as well as some corollaries can be found in Section~\ref{Sect_Proofs}.
Numerous local examples of isolated singularities of affine special K\"ahler structures are contained in~\cites{Haydys15_IsolSing_CMP,CalliesHaydys17_AffSpK2D_Arx}. 
Continuous families of special K\"ahler structures on $\CP^1$ with isolated singularities can be found  in~\cite{CalliesHaydys17_AffSpK2D_Arx}*{Sect.\,3.2}.

The reader may wonder why the holonomy of the flat symplectic connection around the singularity depends on $\b$ only, i.e., on the asymptotic of the \emph{metric}. 
This seems particularly strange in view of the following: Since $\Xi$ can be seen as measuring the difference between $\nabla$ and the Levi-Civita connection $\nabla^{LC}$~\cite{Freed99_SpecialKaehler} and $\Xi$ may have a pole at the origin, one might expect that $\nabla$ and $\nabla^{LC}$ are unrelated in general.  
However, it follows from~\autoref{Thm_spKLocalModels} that this is not the case: The leading terms of $\om_\nabla$ and $\omega_{LC}$ coincide. 
Indeed, this is immediate in the conical case and in the logarithmic one this follows from the observation that the Levi-Civita connection of $g_k$ is $\tfrac 12 (k +\rho^{-1})(\mathbbm 1_2\, d\rho + I_2\, d\theta)$.
In fact, in both cases we have $\om_\nabla - \om_{LC}=o(r^{-1})$ with respect to the flat reference metric on $B_1^*$.
Thus, the singularity of the cubic holomorphic form reflects the singularity of $g$  rather than the deviation of $\nabla$ and $\nabla^{LC}$ near an isolated singularity. 

\medskip
\textsc{Acknowledgements.} Both authors were partially supported by the DFG through SFB 701. 
The second named author was also partially supported by the Simons Collaboration on Special Holonomy in Geometry, Analysis, and Physics.

\section{Preliminaries}
\label{Sect_Preliminaries}

The main purpose of this section is to fix notations.
Details can be found for instance in~\cite{Freed99_SpecialKaehler,Haydys15_IsolSing_CMP}; See also~\cite{CalliesHaydys17_AffSpK2D_Arx} for an elementary exposition.


\paragraph{Conventions.}  Throughout this paper we adopt the convention 
\[
\a \wedge \b = \frac 12 \bigl ( \a\otimes\b - \b \otimes \a \bigr ),
\]
where $\a,\b$ are 1-forms, i.e., we think of differential forms as tensors. 
This convention implies $2\, dx\wedge dy\, (\partial_x, \partial_y)=1$;
In particular, the K\"ahler form of the flat metric $|dz|^2 = dx^2 + dy^2 $ on $\R^2$ is $|dz|^2(I\cdot,\cdot) =2\, dx\wedge dy$.
This also means that $(p,q)$ are Darboux coordinates for a symplectic form $\om\in\Omega^1(\R^2)$, if $\om = 2\, dp\wedge dq$. 

This convention coincides with the one of~\cite{Haydys15_IsolSing_CMP}, but differs from others, which can be found in the literature. 
\medskip

A special K\"ahler structure can be conveniently described locally in terms of special holomorphic coordinates.
The notion of special holomorphic coordinates makes sense in any dimension, but here we restrict ourselves to the case of real dimension two. 
 
Following~\cite{Freed99_SpecialKaehler},  we say that a holomorphic coordinate $Z$ is \emph{special}, if $p=\Re Z$ is flat, i.e., $\nabla dp=0$.
Two special holomorphic coordinates $Z$ and $W$ are said to be conjugate, if  $p$ and $ q:=-\Re W$ are Darboux coordinates for $\om$, i.e., $\om =2\,dp\wedge dq$.
Such coordinates  always exist in a neighborhood of any regular point~\cite{Freed99_SpecialKaehler}.
This does not need to be the case if the point is singular~\cite{CalliesHaydys17_AffSpK2D_Arx}*{Sect.\,2.4}.

One way to construct a pair of conjugate holomorphic coordinates in a neighborhood of some point $m\in M$ is as follows.
Choose local coordinates so as to identify a neighborhood of $M$ with the disc in $\R^2$ of radius $1$.
Choose a pair of covectors $(\a_1,\a_2)$ from $T_mM$ such that $\om_m=2\,\a_1\wedge\a_2$. 
Use parallel transport along radial lines to obtain a pair of one-forms, which are closed, since $\nabla$ is torsion-free. 
Hence, these forms are in fact exact, say $(dp, dq)$, so that $(p,q)$ are Darboux coordinates.
Then $Z$ and $W$ can be found from the relations: $\Re Z =p$ and $\Re W=-q$.

A useful object, which can be attached to a special K\"ahler structure, is the so called holomorphic cubic form, which is defined as follows.
Consider the fiberwise projection $\pi^{(1,0)}$ onto the $T^{1,0}M\subset T_\C M$ as a 1-form with values in $T_{\C}M$. Since this form vanishes on vectors of type $(0,1)$, we can think of $\pi^{(1,0)}$  as an element of $\Om^{1,0}(M; T_\C M)$.
Then, the \emph{holomorphic cubic form} is
\[
\Xi\, :=\, -\om \bigl (\pi^{(1,0)}, \nabla\pi^{(1,0)} \bigr )\in H^0\bigl (M; \mathrm{Sym}^3 T^*M \bigr ).
\]  
In terms of conjugate special holomorphic coordinates $(Z,W)$, \, $\Xi$ can be expressed as follows:
\[
\Xi =\frac{1}{4}\frac{\partial^2 W}{\partial Z^2}\, dZ^3.
\]

Let us now provide a  description of special K\"ahler structures in terms of solutions to a system of nonlinear PDEs, which were first obtained in~\cite{Haydys15_IsolSing_CMP}. 
The analysis of these PDEs is a crucial ingredient of the proof of~\autoref{Thm_spKLocalModels}.   


Write a special K\"ahler metric $g$ on $B_1^*$ in the form 
\[
g\, =\, e^{-u}|d z|^2.
\]
With respect to the  trivialization $(\partial_x, \partial_y)$ of $TB_1^*$  the connection $\nabla$ is described by its connection $1$-form $\om^{}_\nabla\in\Om^1\bigl(\Om;\gl(2,\R)\bigr)$. 
A computation shows that $\nabla$ is torsion-free and satisfies~\eqref{Eq_SpKaehlerCondition} if and only if $\om^{}_\nabla$ can be written in the form
\begin{equation}
\label{Eq_Conn1Form}
\om^{}_\nabla\, = \, 
\begin{pmatrix}
\om^{}_{11} & -*\om^{}_{11}\\
*\om^{}_{22} & \phantom{-* }\om^{}_{22}
\end{pmatrix},
\end{equation}
where $*$ denotes the Hodge operator \wrt the flat metric $|dz|^2=dx^2 + dy^2$.
Furthermore, one can show that  there is a pair $(h, a)\in C^\infty (B_1^*)\times \R$ such that 
\begin{equation}
  \begin{split}\label{Eq_Conn1Components}
		2\om^{}_{11} =  e^u( d h +a\varphi)-d u\qquad\text{and}\qquad
		2\om^{}_{22} =  -e^u(d h +a\varphi)- d u.
  \end{split}
\end{equation}
Here $\varphi =- d\theta$ is a generator of the first de Rham cohomology group of $B_1^*$.
A computation (\cite{Haydys15_IsolSing_CMP}*{Cor 2.3}) shows that $\nabla$ is flat if and only if $(h,u, a)$ satisfies
 \begin{equation}
 \label{Eq_SpKaehlerHU}
 \Delta h\, =\, 0\qquad\text{and}\qquad \Delta u\, =\, |d h + a\varphi|^2 e^{2u}.
 \end{equation}
 Here $\Delta = \partial^2_{xx} +\partial^2_{yy}$ is the (non-positive) Laplace operator \wrt the flat metric of $B_1^*$. 
 
 In terms of solutions of~\eqref{Eq_SpKaehlerHU} the associated holomorphic cubic form can be expressed as follows
 \begin{equation}
   \label{Eq_XiViaH}
  \Xi = \frac{1}{2}\left(\frac{a}{2z} - \frac{\partial h}{\partial z} i\right) dz^3.
 \end{equation}
 
 \begin{remark}
  Note that if $(u,h,a)$ determines an affine special K\"ahler structure with $g=e^{-u} |dz|^2$, then, for $C>0$,  the triple $(u -\log C, Ch, Ca)$ determines a new special K\"ahler structure with the rescaled metric $Cg = C e^{-u}|dz|^2$.
  In particular, the  connection $1$-forms agree and the cubic holomorphic form scales by the same factor $C$ as the metric. 
 \end{remark}

\begin{remark}
	\label{Rem_HUAandChangeOfcoord}
	Observe that the triple $(h, u, a)$ depends on the choice of the local coordinate. 
	We wish to find the triple $(\tilde h, \tilde u, \tilde a)$ corresponding to the new local coordinate $\tilde z = \l\cdot z = e^{i\theta_0}\cdot z$, where $\l$ is a complex number of unit length. 
	Write $h=h_0 + b\log r$, where  $h_0$ is the real part of some holomorphic function $f$, which may have a pole at the origin, and $b\in\R$.
	Then a computation shows that 
	\[
	\tilde h(\tilde z, \bar{\tilde z}) = \Re\bigl ( \l^2 f(\l \tilde z) \bigr) +\Re\bigl ( \l^2 (b+ai)\bigr)\log |\tilde z|,\quad \tilde a = \Im\bigl ( \l^2 (b+ai) \bigr),
	\]
	and $\tilde u = u$.
	
	Notice also that the rescaling $\tilde z = \l \cdot z$ with $\l \in\R_{>0}$ leads to $(\tilde h, \tilde u, \tilde a) = (h, u +2\log \l, a)$. 
\end{remark}

\section{Models of singular special K\"ahler structures}
\label{Sect_Models}

\subsection{Flat cones}
Consider $\C^*$  as a special K\"ahler manifold equipped with the flat conical structure~\eqref{Eq_FlatConicalStr}.
Despite its simplicity, it is instructive to compute special holomorphic coordinates and the holonomy of this structure.

For $r\in (0,1)$ let $\sigma\colon [r,1]\to \C^*$ be the straight line segment,  $\sigma(s) =s$.
A straightforward computation yields that the parallel transport $P_\sigma$ along $\sigma$ is given by
\[
P_\sigma(dx)= r^{-\b/2}\, dx \qquad\text{and} \qquad P_\sigma(dy)= r^{-\b/2}\, dy.
\]
Similarly, if $\gamma_r\colon [0, \psi]\to \C^*$ is the arc of angle $\psi$ starting at the point $r$, then for the parallel transport along $\gamma_r$ we have 
\[
P_{\gamma_r}(dx) = \cos (\tfrac\b 2\psi) \, dx - \sin(\tfrac\b 2\psi)\, dy\qquad\text{and} \qquad
P_{\gamma_r}(dy) = \sin (\tfrac\b 2\psi) \, dx + \cos(\tfrac\b 2\psi)\, dy,
\]
i.e., $P_{\gamma_r}$ is the rotation by the angle $\tfrac \b 2\psi$.
In particular,
\[
\Hol(\rS^1;\nabla) =
\begin{pmatrix}
\cos\pi\b & -\sin\pi\b\\
\sin\pi\b & \cos\pi\b
\end{pmatrix}.
\]

Denote by $P_{\bar{\sigma}}$ the parallel transport from $1$ to $r$ along the segment $[r,1]$. 
The 1-forms $(dp, dq)=\bigl (P_{\gamma_r}\comp P_{\bar{\sigma}}(dx), P_{\gamma_r}\comp P_{\bar{\sigma}}(dy) \bigr )$  are well-defined on $\C\setminus\R_{\ge 0}$. 
Since $\nabla$ is torsion-free, these forms are closed and therefore exact (hence the notation). 
Explicitly, we have 
\[
dp=\Re \bigl ( z^{\b/2 } dz\bigr )\qquad\text{and} \qquad dq = \Im \bigl (z^{\b/2 } dz\bigr ).
\]
Hence, if $\b\neq -2$ we obtain that 
\[
(Z, W)=\frac 2{\b +2}\Bigl (  z^{\tfrac \b 2 +1} ,\  iz^{\tfrac \b 2 +1}\Bigr )
\]
are conjugate special holomorphic coordinates.
For $\b =-2$, we obtain the following pair of conjugate coordinates:
\[
(Z,W) = (\log z, i\log z),
\]  
where $\log z = \log r +i\theta$.


In both cases we have the relation $W=iZ$, which implies that the holomorphic cubic form $\Xi$ vanishes, which is clear anyway, since $\nabla =\nabla^{LC}$. 

Notice that flat cones do not satisfy the hypotheses of~\autoref{Thm_spKLocalModels} since $\Xi$ does not have a finite order at the origin. 
Nevertheless, as we show below they serve as models for special K\"ahler structures, whose associated holomorphic cubic forms have finite orders at isolated singularities.

\subsection{Logarithmic singularities}

In this section we describe the model special K\"ahler structures on $B_1^*$ with logarithmic singularities mentioned in the introduction. 
We first present a more computational approach much in the spirit of the previous subsection. 
A somewhat less computational approach, which also establishes links between $(g_k,\om_{k,\nabla})$ as  $k$ varies, is given afterwards.   

For simplicity of exposition, in the main body of this section we  consider only  the case $C=1=b$ in~\eqref{Eq_ModelLogSing}.
In this case the connection 1-form is given by 
\begin{equation}
\label{Eq_ModelLogSingNormalized}
\begin{aligned}
\omega_{k,\nabla} &= 
\frac 12
\left(
k\rI_2 + 
\begin{pmatrix} 
\sin k\theta & -1+\cos k\theta\\    
1 + \cos k\theta &  -\sin k\theta
\end{pmatrix} \rho^{-1}
\right) d\theta  \\
&+  
\frac 12\left(
k\mathbbm 1_2	+
\begin{pmatrix}
1 - \cos k\theta & \sin k\theta\\
\sin k\theta & 1+\cos k\theta
\end{pmatrix}\rho^{-1} 
\right) d\rho.
\end{aligned}
\end{equation}
The general case is easy to obtain by a suitable choice of $h$ and $a$ below, see~\autoref{Rem_HUAlogModel} for explicit formulae.

\paragraph{First approach: Direct computation.}
	Pick an integer $k\neq 0$ and consider the functions  
	\[
	h(z,\bar z) := \Re\left(\tfrac{z^{k}}{k}\right)\qquad\text{and}\qquad u(z,\bar z):=-k\log|z| - \log|\log|z||, 
	\]
	which are smooth on $B_1^*$.
	Clearly, $h$ is harmonic whereas for $u$ we have
	\[ 
	\Delta u = - \Delta(\log|\log|z||) = \frac{1}{|z|^2(\log|z|)^2}. 
	\]
	Using $dh =  r^{k-1}\cos(k\theta)\, dr - r^{k}\sin(k\theta)\,d\theta$, we obtain $|dh|^2 =  r^{2(k-1)} = |z|^{2(k-1)}$, which yields in turn
	\[ 
	\Delta u = \frac{1}{|z|^2(\log|z|)^2} = |z|^{2(k-1)} \frac{1}{|z|^{2k}(\log|z|)^2} = |dh|^2 e^{2u}.
	\] 
	Therefore  $(h,u,0)$ satisfies~\eqref{Eq_SpKaehlerHU}; Hence, the metric $g = e^{-u} |dz|^2 = -|z|^k\log|z| |dz|^2$ is special K\"ahler.
	
	Using~\eqref{Eq_Conn1Components}, we compute the components of the connection 1-form:
	\begin{align*}
		2\omega_{11} &= - (\log r)^{-1}\left(r^{-1}\cos(k\theta)dr - \sin(k\theta)d\theta\right) + \left(\frac{k}{r} + \frac{1}{r\log r}\right) dr, \\
		2\omega_{22} &=  (\log r)^{-1}\left(r^{-1}\cos(k\theta)dr - \sin(k\theta)d\theta\right) + \left(\frac{k}{r} + \frac{1}{r\log r}\right) dr.
	\end{align*}
This shows that $g=-r^k\log r |dz|^2$ and~\eqref{Eq_ModelLogSingNormalized} is indeed a special K\"ahler structure.

Similarly, for $k=0$ one checks that $h(z,\bar z):=\log|z|$, $u(z):= -\log|\log|z||$, and $a=0$ determine the special K\"ahler structure with the metric $g_0 = -\log|z| |dz|^2$ and connection $1$-form $\omega_{0,\nabla}$ from~\eqref{Eq_ModelLogSingNormalized}.
\medskip

Furthermore, the parallel transport $P_{\sigma}$ of a $1$-form $\eta_1\, dx +\eta_2\, dy$ along the segment $\sigma = [r_0,r_1]\subset \C^*$, where $0<r_0<r_1<1$, is described by the system of ODEs
 \begin{equation*}
 	\dot\eta_1(r) = \frac{k}{2r} \eta_1(r) \qquad\text{and}\qquad
 	\dot\eta_2(r) = \frac{1}{2r}\left(k + \frac{2}{\log r}\right) \eta_2(r),
 \end{equation*}
which can be easily solved explicitly. 
This yields:
\[
P_\sigma(dx)= \left(\frac{r_1}{r_0}\right)^{\frac{k}{2}} dx \qquad\text{and} \qquad P_\sigma(dy)= \left(\frac{r_1}{r_0}\right)^{\frac{k}{2}}\frac{\log r_1}{\log r_0} dy.
\]	

For the  parallel transport along the arc $\gamma_r\colon [0,\theta]\to \C^*$, $\gamma_r(s) = re^{is}$, we obtain the system 
\[ 
\begin{pmatrix} \dot\eta_1(s) \\\dot\eta_2(s)\end{pmatrix} = \tfrac{1}{2\log r}
 \left (
 \begin{pmatrix} 
	 \sin(ks) & \cos(ks)\\ 
	 \cos(ks)& -\sin(ks) 
 \end{pmatrix} - \bigl (k\log r + 1\bigr )\rI_2
 \right )
 \begin{pmatrix} \eta_1(s) \\ \eta_2(s)\end{pmatrix}, 
\] 
One can check that the solution subject to the  initial condition $\eta_1(0)=1, \eta_2(0)=0$ is given by:
	\begin{equation*}
		\eta_1(s) = \cos\left(\tfrac{k}{2} s\right), \qquad
		\eta_2(s) = -\sin\left(\tfrac{k}{2} s\right). 
	\end{equation*}
For the initial condition $\eta_1(0)=0, \eta_2(0) = 1$, the corresponding solution is
	\begin{equation*}
		\eta_1(s) = \sin\left(\tfrac{k}{2} s\right) + \tfrac{s}{\log(r)} \cos\left(\tfrac{k}{2}s\right), \qquad
		\eta_2(s) = \cos\left(\tfrac{k}{2} s\right) - \tfrac{s}{\log(r)} \sin\left(\tfrac{k}{2}s\right). 
	\end{equation*}
	Hence, the holonomy around the origin with a basepoint $r_1$ is given by
	\begin{equation}
	 \label{Eq_HolLogSingEx}
		\Hol(\rS^1_{r_1},\nabla) 
		= (-1)^{k}\begin{pmatrix} 1 & \frac{2\pi}{\log(r_1)} \\ 0 & 1\end{pmatrix}.
	\end{equation}
	
\begin{remark}
	\label{Rem_HolDependsOnr}
 At this point the reader may wonder why the holonomy of $\nabla$ depends on $r_1$. 
 While $\rS^1_{r_1}$ and $\rS^1_{r_2}$ are of course homotopic as free loops in $B_1^*$ for any $r_1,r_2\in (0,1)$,  these loops are based at different points if $r_1\neq r_2$. 
 Therefore, $\Hol(\rS^1_{r_1},\nabla)$  and $\Hol(\rS^1_{r_2},\nabla)$	are \emph{conjugate}, but do not need to be equal.
\end{remark}
	
	Just like in the previous subsection, we compute the parallel transport of $\alpha_1 = r_1^{\frac{k}{2}} dx\in T_{r_1}^*B^*_1$ and $\alpha_2 = -r_1^\frac{k}{2}\log r_1 dy\in T_{r_1}^*B^*_1$ and obtain
	\begin{align*} 
	(dp, dq)
		=\left (\Re\bigl(z^{\frac{k}{2}} dz\bigr),\;-\Im\bigl(\log(z) z^{\frac{k}{2}}dz\bigr)\right ).
	\end{align*}

	 Hence we obtain that for $k\neq -2$ the pair
	\begin{equation}
	\label{Eq_SpecialCoordN}
	\bigl (Z, W\bigr )=\Bigl ( \tfrac{2}{k+2}z^{\frac{k+2}{2}} ,\   -\tfrac{2i}{k+2} z^{\tfrac{k+2}{2}}\left(\log z - \tfrac{2}{k+2} \right)\Bigr )
	\end{equation}
	consists of special conjugate coordinates on $B_1^*\setminus\R_{>0}$. 
		
	Similarly, for $k=-2$, we obtain conjugate coordinates
	\[
	(Z,W) = \Bigl( \log(z),\tfrac{1}{2i}\log(z)^2 \Bigr).
	\]
	
	Furthermore, we have 
	\[ \Xi = -\frac{i}{2}\frac{\partial h}{\partial z} dz^3 = -\frac{i}{4k} \frac{\partial z^{k}}{\partial z} dz^3 = -\tfrac{i}{4} z^{k-1} dz^3. \]
	
	In particular, $\Xi$ has order $n=k-1$ at the origin.

	\begin{remark}\label{Rem_HUAlogModel}
	 If $b$ and/or $C$ in~\eqref{Eq_ModelLogSing} is not necessarily $1$, the triple $(h,u,a)$ is given by
	\begin{align*}
	h(z,\bar z) :=C\Re\left(b~\tfrac{z^{k}}{k}\right), \quad
	u(z,\bar z):= -k\log|z| - \log|\log|z||-\log C,\qand\quad
	a:=0
	\end{align*}
	for $k\neq 0$ and by
	\begin{align*}
	h(z,\bar z) := C\Re(b)\log|z|, \quad
	u(z,\bar z):= - \log|\log|z||-\log C,\qand\quad 
	a:=C\Im(b)
	\end{align*}
	for $k=0$.
	\end{remark}

\paragraph{Second approach: A fundamental example and its pull-backs.}
	Consider the special K\"ahler structure on $B_1^*$ determined by the triple $(u,h,a)$, where
	\begin{equation}
	\label{Eq_HUinFundExample}
	u=\log |z| -\log \bigl | \log|z| \bigr|,\qquad h=-\Re z^{-1} =-\frac x{x^2 + y^2},
	\end{equation}
	and $a=0$. 	
	In other words, the corresponding metric and holomorphic cubic form are given by
	\begin{gather}
		g= -|z|^{-1}\log |z|\; |dz|^2,\label{Eq_FundSpKMetrWithLogSing}\\
		\Xi = -\frac i4 z^{-2}\, dz^3.\label{Eq_FundCubicForm}
	\end{gather}
	This is a special case of the setup of the previous approach  corresponding to $k=-1$.
	
	We claim that a pair of conjugate special holomorphic coordinates on $B_1^*\setminus\R_{>0 }$ is given  by 
	\[
	(Z,W)=\bigl ( 2\sqrt z,- 2i\sqrt z (\log z -2) \bigr).
	\]
	Here $\sqrt z = \sqrt r e^{i\theta/2}$ in polar coordinates.
	
	Let us check that $(Z,W)$ are indeed special coordinates. 
	Denoting
	\[
	p:=2\Re \sqrt z = 2\sqrt r\cos\tfrac \theta 2\quad\text{and}\quad
	q:=-\Re W = - 2\sqrt r \bigl ( \theta\cos \tfrac \theta 2 +\log r  \sin\tfrac \theta 2 -2\sin\tfrac \theta 2 \bigr ),
	\]
	we obtain
	\begin{equation}
	\label{Eq_dpdq}
	\begin{gathered}
	dp = \frac 1{\sqrt r} \bigl (  \cos \tfrac \theta 2\; dr -r\sin\tfrac \theta 2\; d\theta \bigr),\\
	dq = -\frac 1{\sqrt r}  \bigl (  (\theta \cos \tfrac \theta 2 +\log r \sin\tfrac \theta 2)\; dr + r(\log r \cos \tfrac \theta 2 -\theta  \sin\tfrac \theta 2)\; d\theta\bigr).
	\end{gathered}
	\end{equation}
	A straightforward computation yields 
	\[
	dp\wedge dq = -\log r\; dr\wedge d\theta
	= - (r^{-1}\log r)\, rdr\wedge d\theta,
	\]
	which shows that $(p,\,q)$ are Darboux coordinates for the K\"ahler form of~\eqref{Eq_FundSpKMetrWithLogSing}.
	We leave to the reader to check that $(p,q)$ are flat.
	
	It is clear from~\eqref{Eq_dpdq} that the holonomy of the flat symplectic connection $\nabla$ as we go once around a circle centered at the origin is conjugate to 
	\[
	\begin{pmatrix}
	-1 & 2\pi\\
	\phantom{-}0 & -1
	\end{pmatrix}\sim
	\begin{pmatrix}
	-1 & \phantom{-}1\\
	\phantom{-}0 & -1
	\end{pmatrix}.
	\]
Here ``$\sim$" means that the above two matrices are conjugate.

Let $f\colon B_1^*\to B_1^*, \ f(z)=z^{k+2}$, where $k\in\Z,\ k\neq -2$. 
For the pull-back of~\eqref{Eq_FundSpKMetrWithLogSing}, \eqref{Eq_FundCubicForm} we obtain:
	\[
	f^*g =  -(k+2)^3 |z|^{k}\log |z|\;  |dz|^2\quad\text{and}\quad 
	f^*\Xi = -\frac {(k+2)^3 i}4\, z^{k-1}\; dz^3.
	\]
Then clearly $(f^*g, f^*\nabla)$ is again a special K\"ahler structure. 
Thus, we immediately obtain that special holomorphic coordinates are $(f^*Z, f^*W)$, which yields~\eqref{Eq_SpecialCoordN} up to the multiplication by $(k+2)^3$; This factor corresponds to different  normalizations of $f^*g$ and $g_{k}$.

A straightforward albeit laborious  computation shows that the connection 1-form of $f^*\nabla$ is given by~\eqref{Eq_ModelLogSingNormalized}.
In this way one can obtain all structures $(g_{k}, \om_{k,\nabla})$ except for $k=-2$, which must be considered separately just like in the previous approach. 
We omit the details.

\begin{remark}
It is not true that the connection form of $f^*\nabla$ is represented by $f^*\omega_\nabla$ in the sense of~\cite{Haydys15_IsolSing_CMP}. 
The reason is simple: The pair of real coordinates\footnote{$(a,b)$ can be viewed as local coordinates on the $k$-sheeted covering of $B_1^*$.} $(a,b)$, where $a+bi=z^{k+2}=f(z)$, also gives rise to a trivialization of the tangent bundle and $\omega_\nabla$ depends on this trivialization, cf.~\autoref{Rem_HUAandChangeOfcoord}.
Rather, the representation of $f^*\nabla$ can be obtained from $f^*\omega_\nabla$ by applying a gauge transformation.
It is also not true that $(f^*g, f^*\nabla)$ is represented by the triple $(f^*h,f^*u,0)$, where $(h,u)$ are as in~\eqref{Eq_HUinFundExample}.
\end{remark}

\section{Proofs of the main results and some corollaries}
\label{Sect_Proofs}

Write $\Xi=\Xi_0\, dz^3$, where $\Xi_0$ is a holomorphic function on $B_1^*$, and denote $\tilde \Xi_0(z):= z^{-n}\Xi_0(z)$. 
Observe that the order of $\tilde\Xi_0$ at the origin vanishes, i.e., $\tilde\Xi_0 (0)$ is well defined and does not vanish.

\begin{lem}\label{lem_asymptotics_u}
Assume that the holomorphic cubic form $\Xi$ does not have an essential singularity at the origin and denote $n:=\mathrm{ord}_0\,\Xi\in\Z$.
Then there exists $\beta\le n+1$ such that 
\[
\begin{aligned}
u &=-\b \log |z| + v, &&\text{ if } \beta<n+1,\\
u &=-(n+1)\log |z|-\log \bigl |\log |z| \bigr| +\tilde v, &\qquad&\text{ if } \beta=n+1,
\end{aligned}
\]
where the remainder functions $v$ and $\tilde v$ are continuous on $B_1$ and $\tilde v (0) = -2\log 2 - \log |\tilde \Xi_0(0)|$.
Moreover, we have
\begin{equation*}
  \begin{aligned}
  &\partial_x v,\ \partial_yv \text{ are continuous at }0, &\qquad&\text{ if } \b<n+\tfrac 12,\\
  &\partial_x v,\ \partial_yv=O(1), &&\text{ if } \b=n+\tfrac 12,\\
  &\partial_x v,\ \partial_yv=O(|z|^{1-2(\b -n)}), &&\text{ if } \b\in (n+\tfrac 12, n+1),\\
  &\partial_x \tilde v,\ \partial_y\tilde v=O\bigl(|z|^{-1}(\log |z|)^{-2} \bigr), &&\text{ if } \b= n+1.
  \end{aligned}
\end{equation*}
\end{lem}

\begin{proof}
Write $\Xi=\Xi_0\, dz^3$, where $\Xi_0$ is a holomorphic function on $B_1^*$. 
By the proof of Theorem~1.1 of~\cite{Haydys15_IsolSing_CMP}  $u$ satisfies $\Delta u= 16|\Xi_0|^2 e^{2u}$.

Observe that since $\Xi_0$ is holomorphic, $\log |\Xi_0|$ is harmonic. 
Using this, a straightforward computation shows that $u_1:= 2u + 2\log |\Xi_0| + 5\log 2$ is a solution of Liouville's equation $\Delta u_1 = e^{u_1}$. 
With this at hand, the statement of this lemma is immediately obtained from~\cite{Nitsche57_UeberIsolSing}*{Thm.\,1.1}.
\end{proof}

We  say that the order of $h$ at the origin is $\hat N=N+1$ if 
\begin{equation}
  \label{Eq_ExpansionOfHarmFunct}
h=
  \begin{cases}
  \sum\limits_{j=N+1}^\infty p_j(x,y), &\text{if } N+1>0,\\
  q_0\log |z| + h_0, &\text{if } N+1=0,\\
  \sum\limits_{j=1}^{|N+1|} \frac {q_j(x,y)}{|z|^{2j}} +q_0\log |z| + h_0, &\text{if } N+1<0,
  \end{cases}
\end{equation}
where $p_j, q_j$ are homogeneous harmonic polynomials of degree $j$ and  $h_0$ is a smooth harmonic function; cf.~\cite{Axler01_HarmonicFuncTheory}*{p.\,219}.

By~\eqref{Eq_XiViaH}, the order of $\Xi$ at the origin is
\begin{equation}
\label{Eq_nOrdXi}
 n=\mathrm{ord}_0\,\Xi = \begin{cases} N &\text{if } a = 0,\\ \min\{-1,N\} &\text{if }  a\neq 0. \end{cases}
\end{equation}
Notice also that the harmonicity of $h$ implies that $\frac {\partial h}{\partial z}$ is holomorphic.

\begin{proof}[\textbf{Proof of~\autoref{Thm_spKLocalModels}}]  
Let $(h,u,a)$ be a triple representing the special K\"ahler structure $(g,\nabla)$ as in \eqref{Eq_SpKaehlerHU}. As before, denote by $\hat{N} = N+1$  the order of $h$ at $0$.
From \eqref{Eq_ExpansionOfHarmFunct}, we have the following asymptotics:
	\begin{equation}
	  \label{Eq_ExpansionOfH}
	h=\begin{cases}
			\tfrac{r^{N+1}}{N+1}\Re (a_N e^{i{(N+1)}\theta}) + O(r^{N+2}) & \text{if }N\notin \{ -2, -1 \}, \\ 
			-\Re(a_{-2} r^{-1} e^{-i\theta}) + O(\log(r))&\text{if } N=-2,\\
			 \Re(a_{-1} \log(re^{i\theta})) + O(1)&\text{if } N=-1.
		 \end{cases}
	\end{equation}
Note that, while $a_N\in \C$ in general, we have $a_{-1}\in\R$, since $h$ is not allowed to be multivalued.
Furthermore, 
	\begin{equation}
	  \label{Eq_AuxDerivOfH}
		\partial_r h= r^N\Re(a_N e^{i(N+1)\theta}) + O(r^{N+1})
		\quad\text{and}\quad 
		\partial_\theta h= -r^{N+1}\Im(a_N e^{i(N+1)\theta}) + O(r^{N+2}). 
	\end{equation}
Combining~\eqref{Eq_ExpansionOfH},~\eqref{Eq_AuxDerivOfH}, and~\eqref{Eq_XiViaH}, we also obtain
\begin{equation}
	\label{Eq_Xi0Asympt}
	\Xi_0(z)=\begin{cases}
		\tfrac 14 a_N z^N + O(r^{N+1}) & \text{if }N\neq -1, \\ 
		\tfrac 14 (a+ a_{-1}i) z^{-1} + O(1)&\text{if } N=-1.
	\end{cases}
\end{equation}

\noindent
\textbf{Case 1: The conical singularity.}
 Writing $u=-\b\log |z| +v$, with the help of  \autoref{lem_asymptotics_u} we obtain
\begin{equation}\label{Eq_asymptotics_nabla}
\begin{aligned}
 &\bigl | \partial_\theta h \bigr|\lesssim r^{n+1},  && \bigl | e^u\partial_\theta h \bigr|\lesssim r^{n+1-\b},\\
 & \bigl | \partial_\theta u \bigr|=\bigl | \partial_\theta v \bigr|\lesssim r\bigl ( |\partial_x v| + |\partial_y v| \bigr)=o(r^\varepsilon), &\qquad& \bigl | re^u\partial_rh \bigr |\lesssim r^{n+1-\b}\\
 & \bigl | r\partial_r u +\b \bigr|= \bigl | r\partial_r v\bigr|=o(r^\varepsilon),
\end{aligned}
\end{equation}
for a suitable $\varepsilon>0$.

Notice also that if $a\neq 0$, then $n=\mathrm{ord}_0\,\Xi\leq -1$. Therefore, $\b<n+1 \leq 0$, and we obtain
\begin{equation}
  \label{Eq_AuxExpU}
 e^u = e^{-\b\log |z| + v} = |z|^{-\b}e^v = o(r^\varepsilon).
\end{equation}

Furthermore, by~\eqref{Eq_Conn1Components} we have
\begin{equation}
\label{Eq_OmInPolarCoord}
\begin{aligned}
2\omega_{11} &= \bigl (e^u \partial_r h -\partial_r u  \bigr )\, dr + \bigl ( e^u\partial_\theta h -\partial_\theta u - e^u a \bigr )\, d\theta,\\
2\omega_{22} &= -\bigl ( e^u \partial_r h +\partial_r u \bigr )\, dr - \bigl ( e^u \partial_\theta h +\partial_\theta u - e^u a \bigr)\, d\theta.
\end{aligned}
\end{equation}

Substituting  \eqref{Eq_asymptotics_nabla} and \eqref{Eq_AuxExpU} into \eqref{Eq_OmInPolarCoord}, we obtain
\begin{equation*}
 \label{Eq_LimConnForm_gen}
\lim_{|z|\to 0}\left(\omega_\nabla \left(\tfrac{\partial}{\partial\theta}\right)\right)
= \frac{1}{2}\begin{pmatrix} 0& -\beta\\    \beta &  0\end{pmatrix} \qquad\text{and}\qquad \omega_\nabla \left(\tfrac{\partial}{\partial r}\right)
= \begin{pmatrix} \frac{\beta}{2r}& 0\\    0&\frac{\beta}{2r}\end{pmatrix} + o(r^{-1+\varepsilon}).
\end{equation*}

\noindent
\textbf{Case 2: The logarithmic singularity.}   
Just like in Lemma~\ref{lem_asymptotics_u}, write $u=-(n+1)\log r-\log \bigl |\log r \bigr| +\tilde v$. 
The continuity of $\tilde v$ yields
\begin{equation}
 \label{Eq_ExpU}
e^u=-\frac 1{r^{n+1}\log r}e^{\tilde v} = -\frac {C^{-1}  + o(1)}{r^{n+1}\log r},
\end{equation}
where $C = e^{-\tilde v(0)}$ is positive.
Moreover, Lemma~\ref{lem_asymptotics_u} yields
\begin{equation*}
  \label{Eq_AuxDerivOfW}
	\partial_r \tilde v=O\bigl (\tfrac 1{r(\log r)^2}\bigr )\qquad\text{and}\qquad 
	\partial_\theta \tilde v = O\bigl (\tfrac 1{(\log r)^2}\bigr ), 
\end{equation*}
and, therefore, 
\begin{equation}\begin{split}\label{Eq_asympt_du}
 du = & - \frac{n+1}{r}dr - \frac{1}{r\log(r)} dr + \frac{\partial \tilde v}{\partial r} dr + \frac{\partial \tilde v}{\partial \theta} d\theta \\
    = & - \frac{n+1}{r}dr - \frac{1}{r\log(r)} dr + \frac{o(1)}{r\log r} dr + \frac{o(1)}{\log r} d\theta.
\end{split}
\end{equation}

Furthermore, using (\ref{Eq_AuxDerivOfH}) and (\ref{Eq_ExpU}), we obtain
\begin{equation}\begin{split}\label{Eq_asymptotics_eudh}
 e^u(dh + a\varphi) 
		    = & -(C^{-1}+o(1))r^{-(n+1)}\log(r)^{-1}\left( \frac{\partial h}{\partial r} dr + \frac{\partial h}{\partial \theta} d\theta - a d\theta \right) \\
		    = & -\frac{1}{C r\log r} r^{N-n}\left( \Re(a_Ne^{i(N+1)\theta}) + o(1)\right) dr \\
		      & + \frac{1}{C\log r} r^{N-n}\left(\Im(a_Ne^{i(N+1)\theta}) + o(1)\right)d\theta \\
		      & + \frac{1}{C\log r}r^{-(n+1)} a(1+o(1))d\theta.
 \end{split}\end{equation}
Define $b\in\C$  as follows:
\[
b :=\begin{cases} 
				C^{-1}a_n & \text{if }n\neq -1, \\ 
				C^{-1}(a_{-1} + ia) & \text{if }n=-1.
		\end{cases}
\]
Here, we use $a_{-1} = 0$ if $N>-1$. 

We claim that $|b|=1$ in fact.
Indeed, if $n\neq -1$ this follows from the computation  
\[
|b| = |a_n|/C = 4 |\tilde \Xi_0(0)|\, e^{\tilde v(0)} =1,
\]
where the second equality follows by~\eqref{Eq_Xi0Asympt} and the last one by~\autoref{lem_asymptotics_u}.
The case $n=-1$ requires only cosmetic changes in the computation above.

Recalling~\eqref{Eq_nOrdXi}, we obtain from \eqref{Eq_asymptotics_eudh}:
\begin{equation}
\begin{split}\label{Eq_asympt1}
 e^u(dh + a\varphi) = & -\frac{1}{r\log r} \left( \Re(b e^{i(n+1)\theta}) + o(1)\right) dr \\
		      & + \frac{1}{\log r} \left(\Im(b e^{i(n+1)\theta}) + o(1)\right)d\theta.
 \end{split}
\end{equation}

%
%
%
%
%
%
%
%
%
%
%
 Substituting \eqref{Eq_asympt_du} and \eqref{Eq_asympt1} into \eqref{Eq_Conn1Components}, we obtain
\[
 \begin{aligned}
   2\omega_{11} &= \Bigl ( n+1 + \frac {1-\Re(b e^{i(n+1)\theta})}{\log r} +\frac {o(1)}{\log r}\Bigr)\,\frac {dr}r + \Bigl ( \frac{\Im(b e^{i(n+1)\theta})}{\log r} +\frac {o(1)}{\log r} \Bigr)\, d\theta,\\
   2\omega_{22} &= \Bigl ( n+1 + \frac {1+\Re(b e^{i(n+1)\theta})}{\log r} +\frac {o(1)}{\log r}\Bigr)\,\frac {dr}r + \Bigl ( -\frac{\Im(b e^{i(n+1)\theta})}{\log r} +\frac {o(1)}{\log r} \Bigr)\, d\theta.
 \end{aligned}
\]
This immediately yields the statement of this theorem.
\end{proof}

\begin{remark}
	In general, the constant $b$ in~\autoref{Thm_spKLocalModels} depends on the local coordinate and, in most cases, can be normalized to $1$ via suitable choices. 
	For example, if $n\neq -3$ this can be achieved by the change of the local coordinate as in~\autoref{Rem_HUAandChangeOfcoord}.
	However, if $n=-3$, the coefficient of the leading term in~\eqref{Eq_ExpansionOfH} does not depend on the choice of the local coordinate and can not be normalized to $1$. 
\end{remark}

It may seem that~\autoref{Thm_spKLocalModels} implies that the holonomies of  $\nabla$ and its model ($\om_{\b,\nabla}^c$ or $\om_{n+1,\nabla}$) coincide. 
This is not quite accurate,~\autoref{Thm_spKLocalModels} implies only that the closures of the $\SL(2,\R)$--conjugacy classes of the holonomies of $\nabla$ and its local model coincide. 
To explain, recall that $\sigma=\sigma_r$ denotes the straight line segment $[r, 1]$ and $\gamma_r$ denotes the circle of radius $r$ centered at the origin.   
Denoting by $P_r:=P_{\sigma_r}$  and $P_{\gamma_r}$ the parallel transport along $\sigma_r$ and $\gamma_r$ respectively, we have
\begin{equation}
\label{Eq_ParalleTransports}
P_{\gamma_1} = P_r\comp P_{\gamma_r} \comp P_r^{-1}.
\end{equation}
Notice that $P_{\gamma_r}$ does not need to be constant in $r$, see  \autoref{Rem_HolDependsOnr}. 
Moreover, even though $P_{\gamma_r}$ is asymptotic to the parallel transport $ P^0_{\gamma_r}$ of the model connection along $\gamma_r$ as $r\to 0$, one can not pass to the limit in~\eqref{Eq_ParalleTransports} since $P_r$ (as well as $P_{\gamma_r}$) may diverge to infinity.

Observe, however, that  $\tr P_{\gamma_r}$ converges to $\tr P^0_{\gamma_r}$ as $r\to 0$. 
Since $\tr P_{\gamma_r}$ does not depend on $r$ indeed, we must have  $\tr P_{\gamma_1} =\tr P_{\gamma_r}=\tr P^0_{\gamma_r} = \tr P^0_{\gamma_1}$. 
It remains to notice that two matrices in  $\SL(2,\R)$ have equal traces if and only if  the closures of their conjugacy classes are equal; Moreover, the closure of the orbit of $A\in \SL(2,\R)\setminus \{ \pm \mathbbm 1 \}$ is strictly bigger than the orbit of $A$ itself if and only if $\tr A =\pm 2$.  

With this understood, we can proceed to the proof of~\autoref{Cor_Holonomy}.

\medskip

\begin{proof}[\textbf{Proof of~\autoref{Cor_Holonomy}}]
	As explained above,  the closure of the orbit of $\Hol(\rS^1, \nabla)$ in $\SL(2,\R)$ contains
	\[
	H_\b:=\exp
	\begin{pmatrix} 
	0&\pi\beta \\ 
	-\pi\beta& 0
	\end{pmatrix} = 
	\begin{pmatrix} 
	\cos(\pi\beta) &  \sin(\pi\beta) \\ -\sin(\pi\beta) & \cos(\pi\beta)
	\end{pmatrix},\qquad \beta\in (-\infty, n+1].
	\]
	Hence, if $\tr H_\b\neq \pm 2$, then $\Hol(\rS^1, \nabla)$ is conjugate to $H_\b$. 
	If $\tr H_\b=\pm 2$, then $\Hol(\rS^1,\nabla)$ is equal to $\pm\mathbbm 1$ or is conjugate to the elementary Jordan block:
	\[
	\begin{pmatrix}
	\pm 1 & \phantom{\pm}1\\
	0 & \pm 1
	\end{pmatrix}.
	\]
	
	It remains to show, that if $\b\in\Z$ and $\b <n+1$, then $\Hol(\rS^1, \nabla)=\pm\mathbbm 1$.  
	By Theorem~\ref{Thm_spKLocalModels}, the parallel transport of $\nabla$ along the radial segment $[r,1]\subset B_1^*$ is described by the system of ODEs
	\[
	\dot\eta= -\frac 1 s \Bigl ( \frac \b 2 + o(s^\e) \Bigr)\eta, 
	\]
	where $s\in [r, 1]$ and $\eta\colon [r,1]\to\R^2$.
	Notice that the above system becomes singular as $r$ goes to zero.
	
	Writing $\zeta (s) = s^{\b/2}\eta (s)$, we obtain 
	\begin{equation}
	\label{Eq_AuxDiffEqZeta}
	\dot\zeta= o(s^{\e -1})\zeta.
	\end{equation}
	Let us denote $\chi(s):=|\zeta(s)|^2$. 
	Then $\dot{\chi}=2\langle \zeta,\dot\zeta \rangle= o(s^{\e -1})\chi$.
	Integrating this equation, we obtain
	\[
	\log \frac {\chi(s_2)}{\chi(s_1)} = \int\limits_{s_1}^{s_2} o(s^{\e -1})\, ds \ge -C (s_2^{\e} - s_1^{\e}) 
	\]
	for some constant $C>0$ and any $s_1 <s_2$.
	Exponentiating this inequality we obtain
	\[
	\chi (s_1)\le \chi(s_2)\exp\bigl (C(s_2^\e - s_1^\e) \bigr)\le \chi(s_2)\exp (Cs_2^\e).
	\]
	This implies in particular, that for any solution $\zeta$ of~\eqref{Eq_AuxDiffEqZeta} with a fixed initial condition at $s=1$ the function  $\chi = |\zeta|^2$ is bounded on $(0,1)$.

	Think of $P_r$ as a $2\times 2$-matrix. 
	Then the above consideration shows that the matrix-valued function $r^{-\b/2}P_r$ is bounded, hence there is a sequence $r_i\to 0$ such that $r_i^{-\b/2}P_{r_i}$ converges to some $P_0$.
	Moreover, since $\nabla$ is symplectic, we have $\det (r^{-\b/2}P_r)=r^{-\b }\det P_r\to 1$ (recall that the chosen trivialization $(\partial_x, \partial_y)$ is \emph{not} symplectic).
	Hence, we can also assume that $(r_i^{-\b/2}P_{r_i})^{-1}$ converges; In particular, $P_0$ is invertible.
	
	Passing to the limit along the sequence $r_i$ in the equality
	\[
	P_{\gamma_1} = P_{r}\comp P_{\gamma_r}\comp P_r^{-1}= r^{-\b/2}P_r \comp P_{\gamma_r}\comp \bigl (r^{-\b/2}P_r\bigr)^{-1}.
	\]
	and using $P_{\gamma_r}\to \pm\mathbbm 1$, we obtain  $P_{\gamma_1}=\pm\mathbbm 1$.
\end{proof} 

\begin{proposition}
	\label{Cor_IntegralHolonomy}
	$\Hol(\rS^1,\nabla)$ is conjugate to a matrix lying in $\Sp(2,\Z)$ if and only if $\b\in \frac 12\Z\cup\frac 13\Z$.
\end{proposition}
\begin{proof}
	Since $\Hol(\rS^1,\nabla)\in \Sp(2, \R)$, the characteristic polynomial of $\Hol(\rS^1,\nabla)$ has integer coefficients if and only if $\tr\Hol(\rS^1,\nabla)\in\Z$.
	This implies that $\Hol(\rS^1,\nabla)$ is conjugate to a matrix lying in $\Sp(2,\Z)$ if and only if $\cos\pi\b\in \{0, \pm\frac 12, \pm 1\}$. 
\end{proof}

The context where one meets special K\"ahler structures, whose flat symplectic connection $\nabla$ has an integral holonomy, is called algebraic integrable systems.
Roughly speaking, an algebraic integrable system is a smooth \hK manifold $X$ equipped with a holomorphic map $\pi\colon X\to M$, whose fibers are smooth compact Lagrangian tori equipped with  polarizations.
We refer to~\cite{Freed99_SpecialKaehler}*{Def.\,3.1} for details. 
It turns out that an algebraic integrable system induces a special K\"ahler structure on the base $M$~\cite{Freed99_SpecialKaehler}*{Thm.\,3.4}.
Moreover, if the polarization of the fibers is principal, the holonomy of the corresponding flat symplectic connection is integral.  
As a refinement of the above proposition, we obtain the following result. 

\begin{proposition}
	Let $\pi\colon X\to B_1$ be an elliptic fibration such that all fibers $X_z$, $z\neq 0$, are smooth and the central fiber $X_0$ is  not multiple.
	Suppose that $\pi\colon X\setminus X_0\to B_1^*$ admits a structure of the  algebraic integrable system such that all non-central fibers are principally polarized.  
	Let $(g,\nabla)$ be the induced special K\"ahler structure on $B_1^*$.
	Assume that the associated holomorphic cubic form $\Xi$ has a finite order $n\in\Z$ at the origin.
	If the Kodaira type of $X_0$ is as in the left column of Table~\ref{Table_Singularities}, then the type and order multiplied by the factor $2$ of $(g,\nabla)$ is given in the middle and right columns of Table~\ref{Table_Singularities} respectively. 
	In addition, the order of a conical singularity is always smaller than $n+1$.
\begin{table}[H]
	\renewcommand\arraystretch{1.5}
	\begin{center}
		\begin{tabular}{| c | c | c |}
			\hline
			\textbf{Kodaira type} & \textbf{Type of isolated singularity} & \textbf{Twice the order of isolated singularity}  \\ \hline
			$\rI_0$  & conical\ /\  logarithmic & $\b$  even integer\  \;/\ \; $ n+1$,\quad $n$  odd\\
			\hline
			$\rI_0^*$ & conical\ /\  logarithmic & $\b$ odd integer\  \;/\ \; $ n+1$, \quad $n$ even\\
			\hline
			$\rI_b$, \quad $b\neq 0$ & logarithmic & $n+1$,\quad $n$ odd\\
			\hline
			$\rI_b^*$, \quad $b\neq 0$ & logarithmic & $n+1$,\quad $n$ even\\
			\hline
			$\rI\rI$ or $\rI\rI^*$ & conical & $\b = \frac {6k\pm 1}3$,\quad $k\in \Z$\\
			\hline
			$\rI\rI\rI$ or $\rI\rI\rI^*$ & conical  & $\b = \frac 12 +k,\quad k\in\Z$\\
			\hline
			$\rI\rV$ or $\rI\rV^*$ & conical & $\b = \frac {6k\pm 2}3$,\quad $k\in \Z$\\
			\hline
		\end{tabular}
	\caption{\label{Table_Singularities} Singular fibers and corresponding  singularities of affine special K\"ahler structures.}
	\end{center}
\end{table}
\end{proposition}
\begin{proof}
	Let us justify the first row of  Table~\ref{Table_Singularities}; Other rows can be justified in a similar manner.
	
	Thus, assume that the singular fiber is of type $\rI_0$. 
	By the proof of Theorem 11.1 of~\cite{BarthHulek:04} and inspection of Table~6 in the op.cit. we conclude that $\tr \Hol(\rS^1, \nabla)= 2$,  which yields $\cos \pi\b = 1$. 
\end{proof}

\begin{remark}
	It is quite plausible that the statement of the above proposition can be made stronger. 
	For example, it seems likely that in the first two rows of~Table~\ref{Table_Singularities} only the conical singularity can occur. 
	Also, there should be a relation between $b$ and $n$ in the third and fourth rows. 
	More importantly, the assumption that $\pi\colon X\setminus X_0\to B_1^*$ can be extended to an elliptic fibration by adding a suitable central fiber may well follow from the other assumptions.  
	We intend to  address these points elsewhere.
\end{remark}

\appendix

\bibliography{references}
\end{document}